\documentclass[10pt, a4paper]{article}

\usepackage{amsmath, amsthm, amssymb, float}
\usepackage{graphicx,latexsym, amsmath,amsfonts}
\usepackage{amsthm}
\usepackage{hyperref}
\usepackage[nameinlink]{cleveref}
\usepackage{subcaption}
\usepackage{graphicx}
\usepackage{setspace}
\usepackage{makecell}
\usepackage{verbatim}
\usepackage{epstopdf}
\usepackage{color}
\usepackage{enumitem}
\usepackage{a4}
\usepackage[numbers]{natbib}
\allowdisplaybreaks[4]
\hypersetup{colorlinks=true}
\hypersetup{linkcolor=blue}

\numberwithin{equation}{section}
\theoremstyle{plain}

\newtheorem{Theorem}{Theorem}[section]
 
\newtheorem{Lemma}{Lemma}[section]

\newtheorem{Assumption}{Assumption}[section]

\makeatletter \@addtoreset{equation}{section}

\begin{document}

\linespread{1.3}
\pagestyle{plain}

\title
{$\alpha$-scaled strong convergence of stochastic theta method for stochastic differential equations driven by time-changed Lévy noise beyond Lipschitz continuity}

\author{Jingwei Chen\footnotemark[2] \footnotemark[1]}

\renewcommand{\thefootnote}{\fnsymbol{footnote}}
\footnotetext[1]{Corresponding author. (\texttt{chenj22@mails.tsinghua.edu.cn})}
\footnotetext[2]{Yau Mathematical Sciences Center, 
Tsinghua University,
Beijing, 100084, China.}

\date{}

\maketitle
\begin{abstract}
This paper develops an $\alpha$-parametrized framework for analyzing the strong convergence of the stochastic theta (ST) method for stochastic differential equations driven by time-changed Lévy noise (TCSDEwLNs) with time-space-dependent coefficients satisfying local Lipschitz conditions. Properties of the inverse subordinator $E$ are investigated and explicit moment bounds for the exact solution are derived with jump rate incorporated. The analysis demonstrates that the ST method converges strongly with order of $min\{\eta_{F},\eta_{G},\eta_{H},\alpha/2\}$, establishing a precise relationship between numerical accuracy and the time-change mechanism. This theoretical advancement extends existing results and would facilitate applications in finance and biology where time-changed Lévy models are prevalent.


\medskip
\noindent \textbf{AMS Subject Clasification:} 65C30,\;60H10

\end{abstract}

\section{Introduction}
Convergence analysis of the numerical methods is of great importance for SDEs since their analytical solutions may not be written out explicitly.
In recent years, several papers were devoted to the study of convergence of numerical schemes for SDEs driven by time-changed Brownian motion \cite{jum2016strong,liu2020truncated,deng2020semi,2019JinKobayashi,jin2021strong,jingwei2025}.
Among the literatures above, \cite{jingwei2025} is the first to establish an $\alpha$-sensitive framework for time-changed SDEs, proving the convergence order dynamically responds to the stability index $\alpha$ of the stable subordinator. 

This paper investigates the following SDEs driven by time-changed Lévy noise (TCSDEwLNs)
\begin{equation}
    X_{t}=x_{0}+\int_{0}^{t}F(s,X_{s-})dE_{s}+\int_{0}^{t}G(s,X_{s-})dB_{E_{s}}+\int_{0}^{t}\int_{|z|<c}H(s,X_{s-},z)N(dE_{s},dz),\forall t\in[0,T], \label {SDE}
\end{equation}
with the initial value $X_0=x_0\in\mathbb{R}^{d}$,
where
$c>0$,
$X_{t-}:=\underset{s\uparrow t}{lim}X_{s}$,
$T>0$ is a finite time horizon,
$F(t,x):[0,T]\times\mathbb{R}^{d}\rightarrow\mathbb{R}^{d}$,
$G(t,x):[0,T]\times\mathbb{R}^{d}\rightarrow\mathbb{R}^{d\times m}$
and
$H(t,x,z):[0,T]\times\mathbb{R}^{d}\times\mathbb{R}\rightarrow\mathbb{R}^{d}$
are measurable functions with
$F(t,0)=G(t,0)=H(t,0,0)=0$, 
and
$E=\left(E_{t}\right)_{t\geq0}$ is the inverse of an $\alpha$-stable subordinator $D=\left(D_{t}\right)_{t\geq0}$.

We outline the main contributions as follows.
\begin{itemize}
  \item To the best of our knowledge, the study on numerical methods of TCSDEwLNs is very few. We explore stochastic theta (ST) method for \eqref{SDE} when the coefficients satisfy the local Lipschitz condition. Since the coefficients are time-space-dependent, we prove convergence without using duality principle, see \cite{2019JinKobayashi}.
  \item We generalize the moment properties of inverse subordinators $E$ from integer orders in \cite{jingwei2025} to real orders.
  \item We generalize the moment properties of exact solution in \cite{jingwei2025} to include jump rate $\lambda$ as a parameter.
  \item We extend the $\alpha$-parameterized framework in \cite{jingwei2025} to time-changed Lévy noise case, and the strong convergence order is characterized by the Hölder exponent $\eta_H$ of the jump component.
\end{itemize}

\section{Mathematical Preliminaries}
In this paper, 
let $(\Omega,\mathcal{F},\mathbb{P})$ be a complete probability space with filtration $\{\mathcal{F}_{t}\}_{t\geq0}$ satisfying the usual conditions (i.e. right continuous and increasing while $\mathcal{F}_{0}$ contains all $\mathbb{P}$-null sets).
Let $B=\left(B_{t}\right)_{t\geq0}$ be a standard Brownian motion.
Let $D$ be an $\alpha$-stable subordinator with stability index $0<\alpha<1$. A subordinator is a one-dimensional nondecreasing Lévy process with càdlàg paths starting at 0 with Laplace transform
\begin{equation}
  \mathbb{E}[e^{-\xi D_t}] = e^{-t \psi(\xi)}, \xi>0, t\geq 0, \notag
\end{equation}
where the Laplace exponent $\psi:(0,\infty)\rightarrow(0,\infty)$ is $\psi(\xi) = \int_0^\infty (1 - e^{-\xi y}) \Pi(dy), \xi >0$
and the Lévy measure $\Pi$ satisfies $\int_{0}^{\infty}(y\wedge1)\Pi(dy)<\infty.$
We focus on the infinite Lévy measure case, i.e. $\Pi (0,\infty)=\infty$.  
Let $E=(E_{t})_{t\geq0}$ be the inverse of $D$, i.e.
\begin{equation}
    E_{t}:=inf\{u>0:D_{u}>t\},t\geq0. \notag
\end{equation}
If $D$ is a stable subordinator, then $E$ has Mittag-Leffler distributions, see \cite{meerschaert2004limit}.

Let $N$ be a $\mathcal{F}_{t}$-adapted Poisson random measure on $[0,T]\times(\mathbb{R}^{d}-\{0\})$ independent of $B$.
Define a compensated Poisson measure $\tilde{N}(dt,dz):=N(dt,dz)-\lambda \Pi(dz)dt$ 
where $0<\lambda <\infty$ is the jump rate. Let
\begin{equation}
   \int_{0}^{t}\int_{|z|<c}H(s,X_{s-},z)N(dE_{s},dz)=\int_{0}^{t}\int_{|z|<c}H(s,X_{s-},z)\tilde{N}(dE_{s},dz)+\lambda\int_{0}^{t}\int_{|z|<c}H(s,X_{s-},z)\Pi(dz)dE_{s}. \notag
\end{equation}

Let $\mathbb{E}_{B},\mathbb{E}_{D}$ and $\mathbb{E}$ denote the expectation under the probability measures $\mathbb{P}_{B}, \mathbb{P}_{D}$ and $\mathbb{P}$, respectively. 
Suppose $B$ and $E$ are mutually independent, then the product measure satisfies $\mathbb{P}=\mathbb{P}_{B}\times\mathbb{P}_{D}$.

In the content going forward, we assume $m=d=1$ for simplicity of discussions and expressions when necessary; an extension to multidimensional case is straightforward. We denote $C$ as generic positive constants that may change from line to line. We also propose the following assumptions:

\begin{Assumption}(Local Lipschitz Condition).\label{local Lipschitz}
  For every $R>0$, there exists a positive constant $C(R)$, depending only on $R$, such that for all $x,y\in\mathbb{R}^{d}$ and $t\geq 0$,
  \begin{equation}
    \left|F(t,x)-F(t,y)\right|+|G(t,x)-G(t,y)|+\int_{|z|<c}\left|H(t,x,z)-H(t,y,z)\right|\Pi(dz)\leq C(R)|x-y|, \notag
  \end{equation}
  whenever $|x|\vee|y|\leq R$.
\end{Assumption}

\begin{Assumption}\label{polynomial growth}
  There exists constants $C(h)>0$ and $C>0$ such that, for all $t\in[0,T]$ and $x\in\mathbb{R}^{d}$,
  \begin{align}
    \left|F(t,x)\right|\vee|G(t,x)|\leq C(h)\left(1+|x|^{h}\right), \notag
    &&\int_{|z|<c}|H(t,x,z)|^{2}\Pi(dz)\leq K_0\left(1+|x|^{2}\right). \notag
  \end{align}
\end{Assumption} 

\begin{Assumption}\label{monotone}
    There exists a constant $K_1>0$ such that, for all $t\in[0,T]$ and $x\in\mathbb{R}^{d}$,
    \begin{equation}
      \left\langle x,F(t,x)\right\rangle +\frac{2h-1}{2}|G(t,x)|^{2}\leq K_{1}(1+|x|^{2}).  \notag
    \end{equation}
\end{Assumption}

\begin{Assumption}\label{continuity}
  There exist constants $K_2,K_3,K_4>0$ and $\eta_F,\eta_G,\eta_H \in(0,1]$ such that, for all $s,t\in[0,T]$ and $x,y\in\mathbb{R}^{d}$,
  \begin{align} 
    |F(s,x)-F(t,x)|\leq K_2(|1+|x|)|s-t|^{\eta_{F}}, \notag
    &&|G(s,x)-G(t,x)|\leq K_3(|1+|x|)|s-t|^{\eta_{G}},
  \end{align}
  and
  \begin{equation}
    \int_{|z|<c}\left|H(s,x,z)-H(t,x,z)\right|\Pi(dz)\leq K_{4}(|1+|x|)|s-t|^{\eta_{H}}. \notag
  \end{equation}
\end{Assumption}

\begin{Assumption}\label{one-sided Lipschitz}
  There exist a constant $K_5$ such that, for all $s,t\in[0,T]$ and $x,y\in\mathbb{R}^{d}$,
  \begin{align} 
    \left\langle x-y,F((t,x)-F(t,y))\right\rangle \leq K_{5}|x-y|. \notag
  \end{align}
\end{Assumption}

The idea of the proof of the existence and uniqueness of the strong solution to \eqref{SDE} follows from \cite{Protter2004}, \cite{mao1991semimartingales} and Lemma 4.1 of \cite{2011Kobayashi}. The proof is generalized upon reformulating the operators in SDEs driven by semimartingales.

Fix an equidistant step size $\delta\in(0,1)$ and a time horizon $T>0$.
The sample path of the subordinator $D$ is simulated by setting $D_{0}=0$ and $D_{i\delta}:=D_{(i-1)\delta}+Z_{i},i=1,2,3,...,$ with an i.i.d. sequence $\left\{Z_{i}\right\}_{i\in\mathbb{N}}$ distributed as $Z_{i}\overset{d}{=}D_{\delta}.$
One may generate the random variables $Z_{i}$ via the Chambers-Mallows-Stuck (CMS) algorithm \cite{chambers1976method}.
We stop this procedure upon finding the integer $N$ satisfying $T\in[D_{N\delta},D_{(N+1)\delta}).$ 
Next, let 
\begin{equation}
  \tilde{E}_{t}:=\left(min\left\{ n\in\mathbb{N};D_{n\delta}>t\right\} -1\right)\delta,\forall t\in[0,T]. \label{Et}
\end{equation}
We can see that for $n=0,1,2,...,N$, $\tilde{E}_{t}=n\delta$ whenever $t\in[D_{n\delta},D_{(n+1)\delta})$.
Also $\tilde{E}_{t}$ efficiently approximates $E$, i.e. 
\begin{equation}
  E_{t}-\delta\leq\tilde{E}_{t}\leq E_{t},\forall t\in[0,T]. \label{approximates}
\end{equation}
For proofs, see \cite{jum2016strong}.
We can simply approximate $B\circ E$ over the time steps $\{0,\delta,2\delta,...,N\delta\}$.

In order to approximate the solution we consider the ST scheme. Let $\tau_{n}=D_{n\delta},n=0,1,2,...,N$. Given any random step size $\tau_{n+1}-\tau_n,$ define a discrete-time process $(\tilde{X}_{\tau_{n}})$ by setting
\begin{align}
    \tilde{X}_{\tau_{n+1}}	
    &=\tilde{X}_{\tau_{n}}+(1-\theta)F\left(\tau_{n},\tilde{X}_{\tau_{n}}\right)\delta+\theta F\left(\tau_{n+1},\tilde{X}_{\tau_{n+1}}\right)\delta+G\left(\tau_{n},\tilde{X}_{\tau_{n}}\right)\left(B_{(n+1)\delta}-B_{n\delta}\right) \notag \\
    & \quad +\int_{|z|<c}H\left(\tau_{n},\tilde{X}_{\tau_{n}},z\right)\tilde{N}\left(\delta,dz\right)+\lambda\int_{|z|<c}H\left(\tau_{n},\tilde{X}_{\tau_{n}},z\right)\Pi\left(dz\right)\delta, 	\label{discrete STM}
\end{align} 
with $\tilde{X}_{0}=x_{0}$, for $n=0,1,2,...,N-1$,
where $\theta \in[0,1]$. 

Now, we define a continuous-time process $\tilde{X}=\left(\tilde{X}_{t}\right)_{t\in[0,T]}$ by continuous interpolation; i.e. whenver $s\in[\tau_{n},\tau_{n+1}),$
\begin{align}
  \tilde{X}_{s}
  &:=\tilde{X}_{\tau_{n}}+(1-\theta)\int_{\tau_{n}}^{s}F\left(\tau_{n},\tilde{X}_{\tau_{n}}\right)dE_{r}+\theta\int_{\tau_{n}}^{s}F\left(\tau_{n+1},\tilde{X}_{\tau_{n+1}}\right)dE_{r}+\int_{\tau_{n}}^{s}G\left(\tau_{n},\tilde{X}_{\tau_{n}}\right)dB_{E_{r}} \notag \\
  &\quad +\int_{\tau_{n}}^{s}\int_{|z|<c}H\left(\tau_{n},\tilde{X}_{\tau_{n}},z\right)\tilde{N}(dE_{r},dz)+\lambda\int_{\tau_{n}}^{s}\int_{|z|<c}H\left(\tau_{n},\tilde{X}_{\tau_{n}},z\right)\Pi(dz)dE_{r}.
   \label{continuous STM1}
\end{align}
Let $n_{t}=max\left\{ n\in\mathbb{N}\cup\{0\};\tau_{n}\leq t\right\} ,t\geq0$.
Clearly, $\tau_{n_{t}}\leq t\leq\tau_{n_{t}+1}$ for any $t>0$.
By the identity $\tilde{X}_{s}-x_{0}=\sum_{i=0}^{n_{s}-1}\left(\tilde{X}_{\tau_{i+1}}-\tilde{X}_{\tau_{i}}\right)+\left(\tilde{X}_{s}-\tilde{X}_{\tau_{n_{s}}}\right)$, \eqref{continuous STM1} can be rewritten as
\begin{align}
  \tilde{X}_{s}
  &=x_0+(1-\theta)\int_{0}^{s}F\left(\tau_{n_{r}},\tilde{X}_{\tau_{n_{r}}}\right)dE_{r}+\theta\int_{0}^{s}F\left(\tau_{n_{r}+1},\tilde{X}_{\tau_{n_{r}+1}}\right)dE_{r}+\int_{0}^{s}G\left(\tau_{n_{r}},\tilde{X}_{\tau_{n_{r}}}\right)dB_{E_{r}} \notag \\
  &\quad +\int_{0}^{s}\int_{|z|<c}H\left(\tau_{n_{r}},\tilde{X}_{r-},z\right)\tilde{N}\left(dE_{r},dz\right)+\lambda\int_{0}^{s}\int_{|z|<c}H\left(\tau_{n_{r}},\tilde{X}_{r-},z\right)\Pi\left(dz\right)dE_{r}.    \label{continuous STM2}
\end{align}
Define the discrete Forward-Backward Euler-Maruyama (FBEM) scheme
\begin{align}
  \hat{X}_{\tau_{n+1}}
  &=\hat{X}_{\tau_{n}}+F\left(\tau_{n},\tilde{X}_{\tau_{n}}\right)\delta+G\left(\tau_{n},\tilde{X}_{\tau_{n}}\right)\left(B_{(n+1)\delta}-B_{n\delta}\right)+\int_{|z|<c}H\left(\tau_{n},\tilde{X}_{\tau_{n}},z\right)\tilde{N}(\delta,dz) \notag \\
  &\quad +\lambda\int_{|z|<c}H\left(\tau_{n},\tilde{X}_{\tau_{n}},z\right)\Pi(dz)\delta,
  \label{discrete FBEM}
\end{align}
where $\hat{X}_{\tau_{n}}=x_{0}$.
Then the continuous FBEM scheme is 
\begin{align}
  \hat{X}_{s}	
  &=\hat{x}_{0}+\int_{0}^{s}F\left(\tau_{n_{r}},\tilde{X}_{\tau_{n_{r}}}\right)dE_{r}+\int_{0}^{s}G\left(\tau_{n_{r}},\tilde{X}_{\tau_{n_{r}}}\right)dB_{E_{r}}+\int_{0}^{s}\int_{|z|<c}H\left(\tau_{n_{r}},\tilde{X}_{\tau_{n_{r}}-},z\right)\tilde{N}(dE_{r},dz) \notag \\
  &\quad +\int_{0}^{s}\int_{|z|<c}H\left(\tau_{n_{r}},\tilde{X}_{\tau_{n_{r}}-},z\right)\Pi(dz)dE_{r}.
  \label{continuous FBEM}
\end{align}

For each positive number $R$, we define the stopping time $\kappa_{R}:=inf\{t\geq0:|X_{t}|>R\}$, $\vartheta_{R}:=inf\left\{ n\geq0:\left|\tilde{X}_{\tau_{n}}\right|>R\right\}$, $v_{R}:=inf\{t\geq0:|\hat{X}_{t}|>R\ or\ |\tilde{X}_{t}|>R\}$ and $\rho_{R}=\kappa_{R}\wedge\nu_{R}$. Define $f:\mathbb{R}^{d}\rightarrow\mathbb{R}^{d}$ by
\begin{equation}
  f(x)=x-\theta F(t,x)\delta,\forall x\in\mathbb{R}^{d}. \label{f}
\end{equation}

\section{Main results} 

Unless specifically noted, the following theorems and lemmas are directly from or can be deduced following the same ideas in \cite{jingwei2025}.
\begin{Theorem}\label{Theorem:moment of inverse subordinator E}
  Let $E$ be the inverse of an $\alpha$-stable subordinator $D$ with $\alpha\in(0,1)$. Then for any $p\geq 0$ and $0\leq s<t$,
  \begin{equation}
    \mathbb{E}\left[E_{t}^{p}\right]=\frac{\Gamma(p+1)}{\Gamma(\alpha p+1)}t^{\alpha p}.
  \end{equation}
  Furthermore,
  \begin{equation}
    \mathbb{E}\left[\left|E_{t}-E_{s}\right|^{p}\right]=\frac{\Gamma(p+1)}{\Gamma(\alpha p+1)}|t-s|^{p\alpha}. \label{moment property of E-2}
  \end{equation}
\end{Theorem}
\begin{proof}
  Let $S_{\alpha}(x,t)$ be the density of $E$.
  By Proposition 3.2 of \cite{gupta2022inverse}, we have the Mellin transform of the inverse subordinator
  \begin{equation}
        \mathcal{M}\left\{ S_{\alpha}(x,t)\right\} (s)
        =\frac{1}{t^{(1-s)\alpha}}\frac{\Gamma(s)\Gamma((1-s)\alpha)sin((1-s)\alpha\pi)}{\pi}, \notag
  \end{equation}
  which is a analytic function for $\mathcal{R}e(s)>0$ and has simple poles at $s=-k,k=0,1,2....$
  Applying property of Mellin's transform $\mathbb{E}\left[X^{p}\right]=\mathcal{M}\left\{ S_{\alpha}(x,t)\right\}(p+1)$ and Euler's reflection formula $\Gamma(z)\Gamma(1-z)=\frac{\pi}{sin(\pi z)}$ gives
  \begin{equation}
    \mathbb{E}\left[E_{t}^{p}\right]
    =\int_{0}^{\infty}x^{p}S_{\alpha}(x,t)dx 
    =\frac{1}{t^{(1-s)\alpha}}\frac{\Gamma(p+1)\Gamma(-p\alpha)sin(-p\alpha\pi)}{\pi} 
    =\frac{\Gamma(p+1)}{\Gamma(\alpha p+1)}t^{\alpha p}.
  \end{equation}

  The proof of \eqref{moment property of E-2} follows from \cite{jingwei2025}.
\end{proof}

\begin{Theorem}\label{Theorem: exponential moments of powers of inverse subordinator E}
    Let $E$ be the inverse of an $\alpha$-stable subordinator $D$ with $\alpha\in(0,1). $
    Then for $\xi, t>0$,
    \begin{equation}
      \mathbb{E}\left[exp\left(\xi E_{t}\right)\right]=\sum_{k=0}^{\infty}\frac{(\xi t^{\alpha})^{k}}{\Gamma(\alpha k+1)}<\infty. \label{2.6}
    \end{equation}
    For $r>0$, 
    \begin{equation}
      \mathbb{E}\left[exp\left(\xi E_{t}^{r}\right)\right]
      =\sum_{k=0}^{\infty}\frac{\xi^{k}}{k!}\frac{\Gamma(rk+1)}{\Gamma(\alpha rk+1)}t^{\alpha rk}. \label{2.8}
    \end{equation}
    1. If $r<1/(1-\alpha)$, then $\mathbb{E}\left[exp\left(\xi E_{t}^{r}\right)\right]<\infty$.

    2. If $r>1/(1-\alpha)$, then $\mathbb{E}\left[exp\left(\xi E_{t}^{r}\right)\right]=\infty$.
\end{Theorem}

\begin{Theorem}\label{thm: exact pth moment, probability}
  Let Assumptions 2.2-2.3 hold.
  Then, for $p\geq 1$ and any $0\leq t<T$,
  \begin{equation}
    \mathbb{E}\left[\left|X_{t}\right|^{2h}\right]\leq2^{h-1}\sum_{k=0}^{\infty}\frac{\left(\left(2hK_{1}+\lambda K_{0}\right)t^{\alpha}\right)^{k}}{\Gamma(\alpha k+1)}\left(1+\mathbb{E}\left|x_{0}\right|^{2h}\right)
    <\infty, \notag
  \end{equation}
  and
  \begin{equation}
    \mathbb{P}\left(\kappa_{R}\leq T\right)\leq\frac{2^{h-1}}{R^{2h}}\sum_{k=0}^{\infty}\frac{\left(\left(2hK_{1}+\lambda K_{0}\right)t^{\alpha}\right)^{k}}{\Gamma(\alpha k+1)}\left(1+\mathbb{E}\left|x_{0}\right|^{2h}\right). \notag
  \end{equation}
  Furthermore, let $Y_{t}^{(h)}:=1+\underset{0\leq s\leq t}{sup}|X_{s}|^{h}$.
  Then $\mathbb{E}[Y_{t}^{(h)}]<\infty.$
\end{Theorem}

\begin{proof}
  Using the Itô's formula for time-changed Lévy noise (see Lemma 2.6 of \cite{nane2018}), we have
  \begin{align}
    \left(1+\left|X_{t}\right|^{2}\right)^{h}
    &\leq 2^{h-1}\left(1+\left|x_{0}\right|^{p}\right)+2h\int_{0}^{t}\left[1+\left|X_{s}\right|^{2}\right]^{h-1}\left(\left\langle X_{s},F\left(s,X_{s-}\right)\right\rangle +\frac{2h-1}{2}\left|G\left(s,X_{s-}\right)\right|^{2}\right)dE_{s} \notag \\
    &\quad +2h\int_{0}^{t}\left(1+\left|X_{s}\right|^{2}\right)^{h-1}\left\langle X_{s},G\left(s,X_{s-}\right)\right\rangle dB_{E_{s}} \notag \\
    &\quad +\lambda\int_{0}^{t}\int_{|z|<c}\left[\left(1+\left|X_{s}+H\left(s,X_{s},z\right)\right|^{2}\right)^{h}-\left(1+\left|X_{s}\right|^{2}\right)^{h}\right]\Pi(dz)dE_{s} \notag \\
    &\quad +\int_{0}^{t}\int_{|z|<c}\left[\left(1+\left|X_{s-}+H\left(s,X_{s-},z\right)\right|^{2}\right)^{h}-\left(1+\left|X_{s-}\right|^{2}\right)^{h}\right]\tilde{N}(dE_{s},dz).\notag
  \end{align}
  Applying \autoref{polynomial growth} and \autoref{monotone}, we have
  \begin{align}
    \left(1+\left|X_{t}\right|^{2}\right)^{h}
    &\leq 2^{h-1}\left(1+\left|x_{0}\right|^{p}\right)+\left(2hK_{1}+\lambda K_{0}\right)\int_{0}^{t}\left(1+\left|X_{s}\right|^{2}\right)^{h}dE_{s} \notag \\
    &\quad +2h\int_{0}^{t}\left(1+\left|X_{s}\right|^{2}\right)^{h-1}\left\langle X_{s},G\left(s,X_{s}\right)\right\rangle dB_{E_{s}} \notag \\
    &\quad +\int_{0}^{t}\int_{|z|<c}\left[\left(1+\left|X_{s-}+H\left(s,X_{s-},z\right)\right|^{2}\right)^{h}-\left(1+\left|X_{s-}\right|^{2}\right)^{h}\right]\tilde{N}(dE_{s},dz). \notag
  \end{align}
  Taking $\mathbb{E}_{B}$ on both sides, the martingale terms $dB_{E_s}$ and $\tilde{N}(dE_{s},dz)$ are annihilated:
  \begin{align}
    \mathbb{E}_{B}\left[\left(1+\left|X_{t\wedge\kappa_{R}}\right|^{2}\right)^{h}\right]
    &\leq 2^{h-1}\left(1+\left|x_{0}\right|^{2h}\right)+\left(2hK_{1}+\lambda K_{0}\right)\mathbb{E}_{B}\left[\int_{0}^{t}\left(1+\left|X_{s\wedge\kappa_{R}}\right|^{2}\right)^{h}dE_{s}\right]. \notag  
  \end{align}
  By the Gronwall-type inequality (Lemma 6.3 in Chapter IX.6a of \cite{Jacod2003}), taking $\mathbb{E}_{D}$ on both sides, letting $R\rightarrow \infty$ gives
  \begin{equation}
    \mathbb{E}\left[\left(1+\left|X_{t}\right|^{2}\right)^{h}\right]\leq2^{h-1}\left(1+\mathbb{E}\left|x_{0}\right|^{p}\right)\mathbb{E}\left[e^{\left(2hK_{1}+\lambda K_{0}\right)E_{t}}\right]. \notag
  \end{equation}
  Finally, applying \autoref{Theorem: exponential moments of powers of inverse subordinator E}, the theorem is then proved.
\end{proof}

\begin{Lemma}\label{exact any two difference}
  Let \autoref{polynomial growth} hold. 
  Then for $h \geq 1$ and any $0\leq s\leq t$, 
  \begin{equation}
    \mathbb{E}\left[\left|X_{t}-X_{s}\right|\right]
    \leq C(h,\lambda,K_{0})\mathbb{E}\left[Y_{t}^{(2h)}\right]^{1/2}\left\{ \left|t-s\right|^{\alpha}+\left|t-s\right|^{\alpha/2}\right\} , \notag
  \end{equation}
  where $Y_{t}^{(2h)}$ is defined in \autoref{thm: exact pth moment, probability}.
\end{Lemma}
\begin{proof}
  Applying the Cauchy Schwarz inequality and \autoref{polynomial growth}, we have
  \begin{align}
    \mathbb{E}_{B}\left[\left|X_{t}-X_{s}\right|\right]
    &\leq \mathbb{E}_{B}\left[\int_{s}^{t}\left|F(r,X_{r-})\right|dE_{r}\right]+\mathbb{E}_{B}\left[\left|\int_{s}^{t}G(r,X_{r-})dB_{E_{r}}\right|^{2}\right]^{1/2} \notag \\
    &\quad +\mathbb{E}_{B}\left|\int_{s}^{t}\int_{|z|<c}H(r,X_{r-},z)\tilde{N}(dE_{r},dz)\right|+\lambda\mathbb{E}_{B}\left[\left|\int_{s}^{t}\int_{|z|<c}H(r,X_{r-},z)\Pi(dz)dE_{r}\right|^{2}\right]^{1/2} \notag \\
    &\leq C(h)\left(E_{t}-E_{s}\right)\mathbb{E}_{B}\left[Y_{t}^{(h)}\right]+\sqrt{2}C(h)\left(E_{t}-E_{s}\right)^{1/2}\mathbb{E}_{B}\left[Y_{t}^{(2h)}\right]^{1/2} \notag \\
    &\quad +\lambda K_{0}\left(E_{t}-E_{s}\right)^{1/2}\mathbb{E}_{B}\left[Y_{t}^{(2)}\right]^{1/2}. \notag
  \end{align}
  Taking $\mathbb{E}_{D}$ on both sides and using \autoref{Theorem:moment of inverse subordinator E}, we thus have
  \begin{equation}
    \mathbb{E}\left[\left|X_{t}-X_{s}\right|\right]
    \leq C(h,\lambda,K_{0})\mathbb{E}\left[Y_{t}^{(2h)}\right]^{1/2}\left\{ \left|t-s\right|^{\alpha}+\left|t-s\right|^{\alpha/2}\right\}. \notag
  \end{equation}

\end{proof}

\begin{Lemma}\label{f(x)=b existence and uniqueness}
  Let \autoref{monotone} and \autoref{one-sided Lipschitz} hold.
  Let $f$ be defined by \eqref{f}. 
  Then,  for any $b\in\mathbb{R}^{d}$ and $\delta<\delta^{*}=min\left\{1,\frac{1}{K_1\theta}, \frac{1}{K_5\theta}\right\}$, there exists a unique $x\in\mathbb{R}^{d}$ such that
  \begin{equation}
    f(x)=b, \notag
  \end{equation}
  and 
  \begin{equation}
    |x|^{2} \leq \frac{|f(x)|^{2}+2K_{1}\theta\delta}{(1-2K_{1}\theta\delta)}. \notag
  \end{equation}
\end{Lemma}

\begin{Theorem}\label{thm: discrete STM, 2nd moment bound, sup}
  Let Assumptions 2.1-2.3 and 2.5 hold. 
  Then, for any $T>0$ and $\delta<\delta^{*}=min\left\{ 1,\frac{1}{2K_{1}\theta},\frac{1}{K_{5}\theta}\right\}$, there exists a constant $C(T)>0$ such that 
  \begin{equation}
    \underset{\delta<\delta^{*}}{sup}\underset{0\leq\tau_{n}\leq T}{sup}\mathbb{E}\left[\left|\tilde{X}_{\tau_{n}}\right|^{2}\right]<C(T). \notag
  \end{equation}
\end{Theorem}

\begin{Lemma}\label{cor: continuous FBEM-STM, pth moment bound, compact}
  Let Assumptions 2.1-2.5 hold, and $\theta\in[1/2,1]$.
  Then, there exists a constant $C(R,\alpha)$ such that for any $T>0$ and $\delta<\delta^{*}=min\left\{ 1,\frac{1}{2K_{1}\theta},\frac{1}{K_{5}\theta}\right\} $,
  \begin{equation}
    \mathbb{E}\left[\left|\hat{X}_{t}-\tilde{X}_{t}\right|\boldsymbol{1}_{[0,\nu_{R}]}(t)\right]
    \leq C(R,\alpha)\delta^{min\{\eta_{F},1/2\}}. \notag
  \end{equation}
\end{Lemma}

\begin{Theorem} \label{thm: FBEM-STM probability}
  Let Assumptions 2.1-2.5 hold, and $\theta\in[1/2,1]$. 
  Then, for any given $\varepsilon>0$, there exists a positive constant $R_0$ such that for any $R\geq R_0$,
  we can find a positive number $\delta_{0}=\delta_{0}(R)$ such that whenever $\delta <\delta_0$, 
  \begin{equation}
    \mathbb{P}(\nu_{R}\leq T)\leq \frac{\varepsilon}{R^2}, \quad \text{for} \; T>0. \notag
  \end{equation}
\end{Theorem}
\begin{proof}
  By the Itô's formula for time-changed Lévy noise (Lemma 2.6 of \cite{nane2018}) and \autoref{monotone}, we have
  \begin{align}
    \left|\hat{X}_{t\wedge\nu_{R}}\right|^{2}
    &=\left|x_{0}\right|^{2}+2K_{1}\int_{0}^{\tau_{n_{t}}\wedge\nu_{R}}\left(1+\left|\tilde{X}_{\tau_{n_{s}}}\right|^{2}\right)dE_{s}+2\int_{0}^{\tau_{n_{t}}\wedge\nu_{R}}\left\langle \hat{X}_{\tau_{n_{s}}}-\tilde{X}_{\tau_{n_{s}}},F\left(\tau_{n_{s}},\tilde{X}_{\tau_{n_{s}}}\right)\right\rangle dE_{s} \notag \\
    &\quad +2\int_{0}^{\tau_{n_{t}}\wedge\nu_{R}}\left\langle \hat{X}_{\tau_{n_{s}}},G\left(\tau_{n_{s}},\tilde{X}_{\tau_{n_{s}}}\right)\right\rangle dB_{E_{s}}+\lambda K_{1}\int_{0}^{\tau_{n_{t}}\wedge\nu_{R}}\left(1+\left|\hat{X}_{\tau_{n_{s}}}\right|^{2}\right)dE_{s} \notag \\
    &\quad +\int_{0}^{\tau_{n_{t}}\wedge\nu_{R}}\int_{|z|<c}\left[\left|\hat{X}_{\tau_{n_{s}}-}+H\left(\tau_{n_{s}},\hat{X}_{\tau_{n_{s}}-},z\right)\right|^{2}-\left|\hat{X}_{\tau_{n_{s}-}}\right|^{2}\right]\tilde{N}(dE_{s},dz) \notag 
  \end{align}
  Taking $\mathbb{E}_{B}$ on both sides leads to
  \begin{align}
    \mathbb{E}_{B}\left[\left|\hat{X}_{t\wedge\nu_{R}}\right|^{2}\right]
    &\leq \left|x_{0}\right|^{2}+2K_{1}\mathbb{E}_{B}\left[\int_{0}^{\tau_{n_{t}}\wedge\nu_{R}}\left(1+\left|\tilde{X}_{\tau_{n_{s}}}\right|^{2}\right)dE_{s}\right] \notag \\
    &\quad +2\mathbb{E}_{B}\left[\int_{0}^{\tau_{n_{s}}\wedge\nu_{R}}\left|\hat{X}_{\tau_{n_{s}}}-\tilde{X}_{\tau_{n_{s}}}\right|\cdot\left|F\left(\tau_{n_{s}},\tilde{X}_{\tau_{n_{s}}}\right)\right|dE_{s}\right] \notag \\
    &\quad +\lambda K_{1}\mathbb{E}_{B}\left[\int_{0}^{\tau_{n_{t}}\wedge\nu_{R}}\left(1+\left|\hat{X}_{\tau_{n_{s}}}\right|^{2}\right)dE_{s}\right]. \notag
  \end{align}
  Using \autoref{local Lipschitz}, \autoref{thm: discrete STM, 2nd moment bound, sup}, and $F(t,0)=0$ for $t\in[0,\nu_{R}]$, we have
  \begin{equation}
    \left|F\left(\tau_{n_{s}},\tilde{X}_{\tau_{n_{s}}}\right)\right|
    \leq \left|F\left(\tau_{n_{s}},\tilde{X}_{\tau_{n_{s}}}\right)-F\left(\tau_{n_{s}},\tilde{X}_{0}\right)\right|+\left|F\left(\tau_{n_{s}},\tilde{X}_{0}\right)\right|
    \leq C(R). \label{4.2}
  \end{equation}
  Taking $\mathbb{E}_{D}$ on both sides, using \eqref{4.2} and $(a-b+c)^{2}\leq2(|a-b|^{2}+|c|^{2})$, we obtain
  \begin{align}
    \mathbb{E}\left[\left|\hat{X}_{t\wedge\nu_{R}}\right|^{2}\right]
    &\leq\left|x_{0}\right|^{2}+(2+\lambda)K_{1}E_{T}+4K_{1}\mathbb{E}\left[\int_{0}^{\tau_{n_{t}}\wedge\nu_{R}}\left|\tilde{X}_{\tau_{n_{s}}}-\hat{X}_{\tau_{n_{s}}}\right|^{2}dE_{s}\right] \notag \\
    &\quad +(4+\lambda)K_{1}\mathbb{E}\left[\int_{0}^{\tau_{n_{t}}\wedge\nu_{R}}\left|\hat{X}_{\tau_{n_{s}}}\right|^{2}dE_{s}\right]+C(R)\mathbb{E}\left[\int_{0}^{\tau_{n_{s}}\wedge\nu_{R}}\left|\hat{X}_{\tau_{n_{s}}}-\tilde{X}_{\tau_{n_{s}}}\right|dE_{s}\right]. \notag
  \end{align}
  The rest of the proof is similar to the proof of Theorem 4.2 in \cite{jingwei2025}.
\end{proof}

\begin{Theorem}\label{thm: convergence}
  Let Assumptions 2.1-2.5 hold.
  Assume $\theta\in[1/2,1]$, $R=\delta^{-1/2}$ and $\alpha\in(1/2,1)$.
  Then, there exist a constant $C(R,\alpha, \lambda)\in (0,\infty)$ such that
  \begin{equation}
    \mathbb{E}\left[\underset{0\leq t\leq T}{sup}\left|\tilde{X}_{t}-X_{t}\right|\right]\leq C(R,\alpha,\lambda)\delta^{min\{\eta_{F},\eta_{G},\eta_{H},\alpha/2\}}. \notag
  \end{equation}
\end{Theorem}
\begin{proof}
  The proof technique follows the established framework presented in Theorem 4.3 in \cite{jingwei2025}. The estimate for $I_2$ now turns into
  \begin{align}
    \left|\hat{X}_{s\wedge\rho_{R}}-X_{s\wedge\rho_{R}}\right|
    &\leq \left|\int_{0}^{s\wedge\rho_{R}}\left(F\left(\tau_{n_{r}},\tilde{X}_{\tau_{n_{r}}}\right)-F\left(r,X_{r}\right)\right)dE_{r}\right|+\left|\int_{0}^{s\wedge\rho_{R}}\left(G\left(\tau_{n_{r}},\tilde{X}_{\tau_{n_{r}}}\right)-G\left(r,X_{r}\right)\right)dB_{E_{r}}\right| \notag \\
    &\quad +\left|\int_{0}^{s\wedge\rho_{R}}\int_{|z|<c}\left(H(\tau_{n_{r}},\tilde{X}_{\tau_{n_{r}}-},z)-H(r,X_{r-},z)\right)\tilde{N}(dE_{s},dz)\right| \notag \\
    &\quad +\lambda \left|\int_{0}^{s\wedge\rho_{R}}\int_{|z|<c}\left(H(\tau_{n_{r}},\tilde{X}_{\tau_{n_{r}}-},z)-H(r,X_{r-},z)\right)\Pi(dz)dE_{r}\right| \notag \\
    &\leq I_{21}+I_{22}+I_{23}+I_{24}.
  \end{align}
  where 
  \begin{equation}
    I_{21}:=\underset{0\leq s\leq t}{sup}\left|\int_{0}^{s\wedge\rho_{R}}\left(F\left(\tau_{n_{r}},\tilde{X}_{\tau_{n_{r}}}\right)-F\left(r,X_{r}\right)\right)dE_{r}\right|,\notag
  \end{equation}
  \begin{equation}
    I_{22}:=\underset{0\leq s\leq t}{sup}\left|\int_{0}^{s\wedge\rho_{R}}\left(G\left(\tau_{n_{r}},\tilde{X}_{\tau_{n_{r}}}\right)-G\left(r,X_{r}\right)\right)dB_{E_{r}}\right|, \notag
  \end{equation}
  \begin{equation}
    I_{23}:=\underset{0\leq s\leq t}{sup}\left|\int_{0}^{s\wedge\rho_{R}}\int_{|z|<c}\left(H(\tau_{n_{r}},\tilde{X}_{\tau_{n_{r}}-},z)-H(r,X_{r-},z)\right)\tilde{N}(dE_{s},dz)\right|,  \notag
  \end{equation}
  and 
  \begin{equation}
    I_{24}:\underset{0\leq s\leq t}{sup}\left|\lambda \int_{0}^{s\wedge\rho_{R}}\int_{|z|<c}\left(H(\tau_{n_{r}},\tilde{X}_{\tau_{n_{r}}-},z)-H(r,X_{r-},z)\right)\Pi(dz)dE_{r}\right|.\notag
  \end{equation}
  
  For $I_{23}$, it is easy to see that $\mathbb{E}_{B}\left[I_{23}^{2}\right]=0$, since $B$ and $N$ are independent. 
  For $I_{24}$, taking $\mathbb{E}_{B}$ on both sides,
  applying the Burkholder-Davis-Gundy inequality, \autoref{local Lipschitz} and \autoref{continuity}, we have
  \begin{align}
    \mathbb{E}_{B}\left[I_{24}^{2}\right]
    &\leq \mathbb{E}_{B}\left[\underset{0\leq s\leq t}{sup}\left|\lambda\int_{0}^{s\wedge\rho_{R}}\int_{|z|<c}\left(H\left(\tau_{n_{r}},\tilde{X}_{\tau_{n_{r}}-},z\right)-H\left(r,X_{r-},z\right)\right)\Pi(dz)dE_{r}\right|^{2}\right] \notag \\
    &\leq 3b_{2}\lambda^{2}\int_{0}^{t\wedge\rho_{R}}\int_{|z|<c}\mathbb{E}_{B}\left[\left|H\left(\tau_{n_{r}},\tilde{X}_{\tau_{n_{r}}-},z\right)-H\left(\tau_{n_{r}},X_{\tau_{n_{r}}-},z\right)\right|^{2}\right]\Pi(dz)dE_{r} \notag \\
    &\quad +3b_{2}\lambda^{2}\int_{0}^{t\wedge\rho_{R}}\int_{|z|<c}\mathbb{E}_{B}\left[\left|H\left(\tau_{n_{r}},X_{\tau_{n_{r}}-},z\right)-H\left(\tau_{n_{r}},X_{r-},z\right)\right|^{2}\right]\Pi(dz)dE_{r} \notag \\
    &\quad +3b_{2}\lambda^{2}\int_{0}^{t\wedge\rho_{R}}\int_{|z|<c}\mathbb{E}_{B}\left[\left|H\left(\tau_{n_{r}},X_{r-},z\right)-H\left(r,X_{r-},z\right)\right|^{2}\right]\Pi(dz)dE_{r} \notag \\
    &\leq 3b_{2}\lambda^{2}C(R)\int_{0}^{t\wedge\rho_{R}}\mathbb{E}_{B}\left[\left|\tilde{X}_{\tau_{n_{r}}}-X_{\tau_{n_{r}}}\right|^{2}\right]dE_{r}+3b_{2}\lambda^{2}C(R)\int_{0}^{t\wedge\rho_{R}}\mathbb{E}_{B}\left[\left|X_{\tau_{n_{r}}}-X_{r}\right|^{2}\right]dE_{r} \notag \\
    &\quad +3b_{2}K_{4}^{2}\lambda^{2}\int_{0}^{t\wedge\rho_{R}}\mathbb{E}_{B}\left[\left(1+\left|X_{r-}\right|\right)^{2}\left|\tau_{n_{r}}-r\right|^{2\eta_{H}}\right]dE_{r} \notag \\
    &\leq C(R,\lambda)\int_{0}^{t\wedge\rho_{R}}\mathbb{E}_{B}\left[Z_{r}^{2}\right]dE_{r}+C(R,\lambda)E_{T}\mathbb{E}_{B}[Y_{T}^{(2h)}]\delta^{min\{2\eta_{H},\alpha\}}.
  \end{align}
  Applying the Gronwall-type inequality (Lemma 6.3 in Chapter IX.6a of [37]), taking $\mathbb{E}_D$, and Theorem 3.1-3.3 gives
  \begin{equation}
    \mathbb{E}\left[I_{2}\right]
    \leq C(R,\alpha,\lambda)\delta^{min\{\eta_{F},\eta_{G},\eta_{H},\alpha/2\}}.
  \end{equation}
  Putting together $I_1, I_2, I_3, I_4$ and $I_5$ we obtain the desired result.
\end{proof}

\bibliographystyle{elsarticle-num}  
\bibliography{main}

\end{document}